\documentclass{amsart}

\usepackage{helvet, color}

\usepackage{amscd,amsmath,amsxtra,amsthm,amssymb,stmaryrd,xr,mathrsfs,mathtools,enumerate}
\usepackage[all,cmtip]{xy}

 \DeclareFontFamily{U}{wncy}{}
    \DeclareFontShape{U}{wncy}{m}{n}{<->wncyr10}{}
    \DeclareSymbolFont{mcy}{U}{wncy}{m}{n}
    \DeclareMathSymbol{\Sh}{\mathord}{mcy}{"58}

\newtheorem{theorem}{Theorem}[section]
\newtheorem{lemma}[theorem]{Lemma}

\newtheorem{proposition}[theorem]{Proposition}

\numberwithin{equation}{section}

\theoremstyle{remark}
\newtheorem{remark}[theorem]{Remark}

\newcommand{\Z}{\mathbf{Z}}

\newcommand{\Q}{\mathbf{Q}}
\newcommand{\F}{\mathbf{F}}

\begin{document}

\title[Two infinite families with rank $>1$]{Two infinite families of elliptic curves with rank greater than one}

\author[J.~Hatley]{Jeffrey Hatley}
\address[Hatley]{
Department of Mathematics\\
Union College\\
Bailey Hall 202\\
Schenectady, NY 12308\\
USA}
\email{hatleyj@union.edu}

\author[J.~Stack]{Jason Stack}
\address[Stack]{
Department of Mathematics\\
Union College\\
Bailey Hall 202\\
Schenectady, NY 12308\\
USA}
\email{stackj@union.edu}

\begin{abstract}
We prove, using elementary methods, that each member of the infinite families of elliptic curves given by $E_m
\colon y^2=x^3 - x + m^6$ and $E_m'
\colon y^2=x^3 + x - m^6$ have rank at least $2$ and 3, respectively, under mild restrictions on $m$. We also prove stronger results for $E_m$ and $E_m'$ using more technical machinery.
\end{abstract}



\maketitle

\section{Introduction}\label{section:intro}

The celebrated Mordell-Weil theorem states that the $\Q$-rational points on an elliptic curve $E/\Q$ form a finitely-generated abelian group, so that we have a group isomorphism
\begin{equation*}
E(\Q) \simeq \Z^r \oplus \Z_{\mathrm{tor}}
\end{equation*}
for some integer $r \geq 0$ (called the \textit{rank} of $E$) and for some finite abelian group $\Z_{\mathrm{tor}}$ (called the \textit{torsion subgroup} of $E$). Both the rank and the torsion subgroup have been the focus of much research for many decades. For instance, Mazur \cite{Mazur-ideal} classified all possible torsion subgroups that can occur, and more recently, Harron and Snowden \cite{HarronSnowden} computed the frequency with which these torsion subgroups appear.

The rank is much less well-understood. For instance, it is unknown whether elliptic curves with arbitrarily large rank exist. The largest known rank is currently $28$, found by Noam Elkies \cite{dujella}. For many years it was a folklore conjecture that the ranks of elliptic curves are unbounded, but recent heuristics have started to cast some doubt; for example, work of Park, Poonen, Voight, and Wood \cite{PPVW} suggests that all but finitely many curves have rank at most $21$.

On the other hand, it has long been expected that the ``average rank'' of elliptic curves (appropriately counted) is $1/2$, meaning that half of all elliptic curves have rank $0$, half have rank $1$, and that curves with rank $\geq 2$ have density zero and thus are exceedingly rare. Indeed, this conjecture has been bolstered by a recent probabilistic model due to Lozano-Robledo \cite[Theorem 1.2]{Alvaro-heuristic}), and the average rank has been proven to be at most $7/6$ by Bhargava and Shankar \cite{BhargavaShankhar}.

It follows that examples of elliptic curves with ``large rank'' (i.e. $r \geq 2$) are inherently interesting. In this paper, we exhibit two infinite families of elliptic curves with large rank, and we prove that they possess this property by direct and elementary means. In particular, we first prove the following.

\begin{theorem}\label{thm:main-intro-rk3}
Let $m>1$ be an odd integer divisible by $3$. Then the elliptic curve
\[
E_m \colon y^2 = x^3 - x + m^6
\]
has rank at least $3$.
\end{theorem}

\begin{remark}
When $m=0$ or $m=1$, our methods do not apply, and indeed, one may check that the corresponding ranks are $0$ and $1$, respectively.
\end{remark}

Our approach is inspired by similar work of Brown and Myers, who studied the curves given by
\[
y^2 = x^3-x+m^2.
\]
In \cite{BrownMyers}, they proved that each curve in this family has rank at least $2$, and that when $m=54n^2 - 165n-90$, the curve has rank at least $3$. Our theorem therefore produces an even simpler-to-describe subfamily of curves of rank at least $3$. We prove this theorem in Section \ref{sec:family1}.

Our second result studies a similar-looking family of elliptic curves and obtains a result in the spirit of \cite[Theorem 1]{BrownMyers}.

\begin{theorem}\label{thm:main-intro-rk2}
Let $m>1$ be an integer such that $3 \nmid m$ and $2 \mid \mid m$. Then the elliptic curve
\[
E'_m \colon y^2 = x^3 + x - m^6
\]
has rank at least $2$.
\end{theorem}

We note that $2 \mid \mid m$ means that $m \equiv 2 \mod 4$. The proof of this theorem occupies Section \ref{sec:family2}.

\begin{remark}
Once again, when $m=0$ or $m=1$, our methods do not apply, and one may check that the corresponding ranks are $0$ and $1$, respectively.
\end{remark}

Both of these theorems were discovered computationally by investigating ranks of elliptic curves in families using SageMath \cite{sage}, and they both yield to proofs by elementary means.

However, the computational data suggests that, in fact, Theorems  \ref{thm:main-intro-rk3} and \ref{thm:main-intro-rk2} are true for any $m>1$, and we are able to prove this as well, using less elementary techniques. Namely, we prove

\begin{theorem}\label{thm:unconditional}
Let $m > 1$ be an integer. Then
\begin{enumerate}
\item the elliptic curve $E_m$ has rank at least $3$, and
\item the elliptic curve $E'_m$ has rank at least $2$.
\end{enumerate}
\end{theorem}

This last theorem, which is clearly a strengthened version of Theorem \ref{thm:main-intro-rk2}, does not seem accessible by the same techniques that Brown and Myers employed, but we are able to prove it through a simple and explicit computation of a N\'{e}ron-Tate height pairing matrix. We do this in Section \ref{sec:height-matrix}.

\subsection{Strategy}

The strategy for proving Theorems \ref{thm:main-intro-rk3} and \ref{thm:main-intro-rk2}, which is the same as in \cite{BrownMyers}, consists of two main steps. The first step is to prove that our curves have trivial torsion subgroup, and the utility of this fact is given by the following proposition.

\begin{proposition}\label{prop:utility-of-trivial-torsion}
Let $E/\Q$ be an elliptic curve with trivial torsion subgroup. Then
\[
|E(\Q)/2E(\Q)| = 2^r
\]
where $r$ is the Mordell-Weil rank of $E(\Q)$.
\end{proposition}
\begin{proof}
This is \cite[Theorem 4]{BrownMyers}. The main point is that the Mordell-Weil isomorphism
\begin{equation*}
E(\Q) \simeq \Z^r \oplus \Z_{\mathrm{tor}}
\end{equation*}
reduces, under the hypothesis of the proposition, to the isomorphism
\begin{equation*}
E(\Q) \simeq \Z^r.
\end{equation*}
This in turn induces an isomorphism of $2$-groups
\begin{equation*}
E(\Q)/2E(\Q) \simeq (\Z/2\Z)^r.
\end{equation*}
\end{proof}

Given this proposition, it will not surprise the reader to learn that the second step of our strategy is to produce some explicit points on each elliptic curve $E$ and to show that they are independent in the group $E(\Q)/2E(\Q)$. Our curves are chosen so that rational points can be found easily by inspection, and then we use arguments from elementary number theory to prove that these points are not obtained by doubling any other rational points.

Since we will use it frequently in the rest of the paper, we remind the reader that if $P=(x_0,y_0)$ is a rational point on the elliptic curve
\[
y^2=x^3+ax+b,
\]
then the $x$-coordinate of $2P=P+P$ is given by
\begin{equation}\label{eq:doubling-formula}
x(2P) = \left(\frac{3x^2_0+a}{2y_0}\right)^2 -2x_0.
\end{equation}

In order to tie together the ideas described above, we need the following result, which is well-known and appears as \cite[Lemma 8]{BrownMyers}. We include a proof for completeness and to emphasize its elementary nature. In what follows, we write $\mathcal{O}$ for the point at infinity of $E(\Q)$, which plays the role of the identity element for its group structure. Also, given a point $P \in E(\Q)$, we will write $[P]$ for its image in $E(\Q)/2E(\Q)$.

\begin{proposition}\label{prop:non-doubles-contribute-to-rank}
Let $E$ be an elliptic curve with trivial $2$-torsion, and consider a set of rational points $$\{P_1,\ldots,P_k, Q\} \subset E(\Q).$$ Suppose the points $P_i$ are independent in $E(\Q)$. If $Q$ is not in the subgroup $$\langle [P_1], \ldots, [P_{k}] \rangle \subset E(\Q)/2E(\Q)$$ of $E(\Q)/2E(\Q)$ generated by the points $P_i$, then the points $$P_1,\ldots,P_k,Q$$ are independent in $E(\Q)$. In particular, the rank of $E(\Q)$ is at least $k+1$.
\begin{proof}
Suppose on the contrary we have
\begin{equation}\label{eq:linear-dependence}
\mathcal{O} =n_0 Q +n_1 P_1 + \cdots + n_{k}P_{k}
\end{equation}
where not all of the $n_i$ are zero. We may assume without loss of generality that $n_0$ is positive and minimal among all such relations.

If $n_0$ is odd, then \eqref{eq:linear-dependence} induces a relation
\[
[Q] = [n_1 P_1 + \cdots +n_{k}P_{k}]
\]
in $E(\Q)/2E(\Q)$, contradicting the assumption that $Q$ is not in the subgroup generated by the $P_i$.

If instead $n_0$ is even, then \eqref{eq:linear-dependence} induces the relation
\begin{equation}\label{eq:dependent-mod-2}
[\mathcal{O}] = [n_1 P_1 + \cdots+ n_{k}P_{k}]
\end{equation}
in $E(\Q)/2E(\Q)$. Since the $P_i$ are independent in $E(\Q)$, \eqref{eq:dependent-mod-2} implies that each of the $n_i$ is even. Thus for each $0 \leq i \leq k$ we may write $n_i=2m_i$ for some integer $m_i$, and upon substituting and factoring, equation \eqref{eq:linear-dependence} becomes
\begin{equation}\label{eq:dependence-torsion}
2(m_0 Q + m_1 P_q + \cdots +m_k P_k) = \mathcal{O}.
\end{equation}
That is, the point
\[
m_0 Q + m_1 P_q + \cdots +m_k P_k
\]
is $2$-torsion. But $E$ has only trivial torsion, so \eqref{eq:dependence-torsion} implies that
\[
\mathcal{O}=m_0 Q + m_1 P_q + \cdots +m_k P_k
\]
in $E(\Q)$. Since $m_0 < n_0$, this contradicts the minimality of $n_0$, which concludes the proof of independence. The last statement of the theorem now follows from Proposition \ref{prop:utility-of-trivial-torsion}.
\end{proof}

\end{proposition}

\subsection*{Acknowledgements}
Much of the work in this paper was carried out by the second-named author while he was an undergraduate. The authors are grateful to Union College for its support of undergraduate research. The authors also thank the anonymous referee for their careful reading of an earlier version of this paper, as well as their helpful corrections and suggestions.

\section{Torsion subgroups}\label{sec:torsion}

In this section we will prove that each elliptic curve in the families
\[
E_m \colon y^2= x^3 - x + m^6
\]
and
\[
E'_m \colon y^2=x^3 + x - m^6
\]
has trivial torsion subgroup. In fact, for $E_m$ this was already established by Brown and Myers, as we now record. Their proof follows the same essential strategy as the proof of our Theorem \ref{thm:torsion-family-two} below.

\begin{theorem}\label{thm:torsion-family-one}
For every $m>1$, the elliptic curve
\[
E_m \colon y^2 = x^3 - x + m^6
\]
has trivial torsion subgroup.
\end{theorem}
\begin{proof}
In \cite[Theorem 3]{BrownMyers}, Brown and Myers prove that for $m>1$, each elliptic curve in the family
\begin{equation}\label{eq:brown-myers-family}
\mathcal{E}_m \colon y^2=x^3 + x - m^2
\end{equation}
has trivial torsion subgroup. Since we have $E_m=\mathcal{E}_{m^3}$, their result immediately applies to our (sub)family of curves $E_m$.
\end{proof}

Recall that a rational prime $p$ is called a \textit{prime of good reduction} for an elliptic curve $E/\Q$ if it does not divide the discriminant $\mathrm{disc}(E)$ of $E$. Brown and Myers prove their result on the torsion subgroup by using the well-known fact that, for each prime $p$ of good reduction for $E$, there is an injection
\[
E(\Q)_{\mathrm{tors}} \hookrightarrow E(\F_p)
\]
of the torsion subgroup into the group of $\F_p$-points on $E$, where $\F_p$ denotes the finite field with $p$ elements. This is a very common strategy, and we will now use it to prove the analogous result for the family $E'_m$.

\begin{theorem}\label{thm:torsion-family-two}
For every nonzero integer $m$, the elliptic curve
\[
E'_m \colon y^2 = x^3 + x - m^6
\]
has trivial torsion subgroup.
\end{theorem}
\begin{proof}
The discriminant of $E'_m$ is given by
\[
\mathrm{disc}(E'_m)=-16(4+27m^{12}).
\]
We see immediately that $\mathrm{disc}(E'_m) \equiv 2 \mod 3$, so $3$ is a prime of good reduction for $E'_m$. If $3 \nmid m$, then by Fermat's little theorem we have $m^6 \equiv 1 \mod 3$, and over $\F_3$ the defining equation becomes
\[
y^2 = x^3 + x - 1,
\]
from which it is easy to verify that
\[
E'_m(\F_3)=\{ \mathcal{O}, (1,1), (1,2), (2,0)\}.
\]
If $3 \mid m$, then
over $\F_3$ the defining equation becomes
\[
y^2 = x^3 + x,
\]
and we have
\[
E'_m(\F_3)=\{ \mathcal{O}, (0,1), (2,1), (2,2)\}.
\]
So in any case, we see that$|E'_m(\F_3)|=4$.
Since $E'_m(\Q)_\mathrm{tors}$ injects into $E(\F_3)$, it follows that if $E'_m(\Q)_\mathrm{tors}$ is nontrivial, then it must contain a point of order $2$.

If $E'_m$ has a point of order $2$, then by the Nagell-Lutz theorem, this point must be of the form $P=(a,0)$ with $a$ an integer, so we have
\[
0 = a^3 +a-m^6.
\]
This is clearly impossible when $m=\pm1$, so instead assume $|m|>1$, in which case it follows that $a > 1$. Rewriting the above equation, we have
\[
m^6 = a(a^2+1).
\]
The factors on the right-hand side are relatively prime, so by the Fundamental Theorem of Arithmetic, each is a $6$th power. In particular, there must exist an integer $c>1$ such that
\[
a^2+1 = c^6,
\]
or equivalently
\[
1=(c^3-a)(c^3+a).
\]
Since both factors on the right are integers and $c^3+a>1$, this is impossible. So $E'_m(\Q)$ contains no points of order $2$, and we conclude that $E'_m(\Q)_\mathrm{tors}$ is trivial.
\end{proof}

Thus, in light of Proposition \ref{prop:non-doubles-contribute-to-rank}, our task is now to produce many independent points in $E_m(\Q)/2E_m(\Q)$ and $E'_m(\Q)/2E'_m(\Q)$. We do this in the next two sections, respectively.

\section{The family $E_m$ has rank $\geq 3$}\label{sec:family1}

In this section, let us consider the family of elliptic curves
\[
E_m \colon y^2 = x^3-x+m^6, \quad \text{where $m$ is odd and $3 \mid m$}.
\]
We first observe that $E_m(\Q)$ contains the following three points:
\[
P=(0,m^3), \quad  Q=(-1,m^3), \quad \text{and} \quad  R=(-m^2,m),
\]
Our strategy is to show that $P$, $Q$, and $R$
are independent in $E(\Q)$ by showing that they generate a subgroup of order $8$ in $E_m(\Q)/2E_m(\Q)$. This will be accomplished by showing that they satisfy the conditions of Proposition \ref{prop:non-doubles-contribute-to-rank}, from which we may conclude that the rank of $E_m$ is at least 3.

As we observed earlier, our family $E_m$ is a subfamily of the Brown-Myers family $\mathcal{E}_m$ as described in \eqref{eq:brown-myers-family}. We thus get the following result for free.

\begin{proposition}\label{prop:family1-prop1}
The points $P$ and $Q$ generate a subgroup of order $4$ in $E_m(\Q)/2E_m(\Q)$.
\end{proposition}
\begin{proof}
This is \cite[Lemma 6]{BrownMyers} applied to their curve $\mathcal{E}_{m^3}$.
\end{proof}

It remains to study the points
\begin{align*}
R&=(-m^2,m)\\
P+R&=\left(\frac{2m^4-2m^2+1}{m^2},\frac{-3m^6+4m^4-3m^2+1}{m^3}\right)\\
Q+R&=(2m^2+1,-3m^3-2m)\\
P+Q+R&=(2m^2-1,3m^3-2m)
\end{align*}

Since we have assumed $m$ is an odd integer, the $x$-coordinates of $R$, $Q+R$, and $P+Q+R$ are all odd integers. We have the following lemma.

\begin{lemma}\label{lem:family-one-odd-k}
Let $A = (k,l)$ and $B = (x_0, y_0)$ belong to $E_m(\Q)$. If $k$ is an odd integer, then $A \neq 2B$.
\end{lemma}
\begin{proof}
Without loss of generality, we may write $x_0=\frac{u}{r^2}$ where $u,r, \in \Z$ with $\gcd(r,u)=1$. Suppose $A = 2B$,  so that $k = x(2B)$. From the doubling formula for points on elliptic curves as recalled in equation \eqref{eq:doubling-formula}, we obtain the equality
\begin{equation}\label{eq:k-equation}
4kr^2(u^3 - ur^4 + m^6r^6) = u^4 + 2u^2r^4 + r^8 - 8um^6r^6.
\end{equation}
Suppose $r \neq \pm 1$. Then there exists some prime $p$ which divides $r$. Considering \eqref{eq:k-equation} modulo $p$ then gives the congruence
\[
0 \equiv u^4 \mod p
\]
which implies $p|u$, but this contradicts our assumption that $\gcd(u,r) = 1$, so we conclude that $r = \pm 1$. Equation \eqref{eq:k-equation} then simplifies to
\begin{equation}\label{eq:k-equation-with-r-equal-to-1}
4k(u^3-u+m^6) = (u^2+1)^2-8um^6.
\end{equation}
Now recall that we have assumed $3$ divides $m$, so when we consider \eqref{eq:k-equation-with-r-equal-to-1} modulo 3, we obtain
\begin{equation}\label{eq:mod3-first-step}
ku(u- 1)(u+1) \equiv (u^2+1)^2 \mod 3.
\end{equation}
The $u$-factors reveal that the left-hand side of \eqref{eq:mod3-first-step} is identically $0$ modulo $3$. On the other hand, the right-hand side is always $1$ modulo $3$. This is impossible, and thus we have shown $A \neq 2B$.
\end{proof}

The following proposition is then an immediate corollary.

\begin{proposition}\label{prop:family1prop2}
The points $R$, $Q+R$, and $P+Q+R$ are nontrivial in $E_m(\Q)/2E_m(\Q)$.
\end{proposition}

We can obtain a similar result for the rational point $P+R$ using the work of Brown and Myers.

\begin{proposition}\label{prop:family1prop3}
The point $P+R$ is nontrivial in $E_m(\Q)/2E_m(\Q)$.
\end{proposition}
\begin{proof}
As stated above, the $x$-coordinate of $P+R$ is
\[
x(P+R)=\frac{2m^4-2m^2+1}{m^2}.
\]
Since we assume $m$ is odd, the denominator of $x(P+R)$ is odd, and the numerator of $x(P+R)$ is congruent to $1$ mod $4$. The result now follows from \cite[Theorem 5(c)]{BrownMyers}.
\end{proof}

We may now prove Theorem \ref{thm:torsion-family-one}.

\begin{proof}[Proof of Theorem \ref{thm:torsion-family-one}]

By Proposition \ref{prop:family1-prop1}, the points $P$ and $Q$ are independent in $\Q$ and generate a subgroup $\langle[P],[Q]\rangle$ of order $4$ in $E_m(\Q)/2E_m(\Q)$. By Propositions \ref{prop:family1prop2} and \ref{prop:family1prop3}, $R$ does not belong to $\langle [P], [Q] \rangle$, so applying Proposition \ref{prop:utility-of-trivial-torsion} shows that the rank of $E_m(\Q)$ is at least $3$.
\end{proof}

\section{The family $E'_m$ has rank $\geq 2$}\label{sec:family2}

In this next section, let us consider the family of elliptic curves
\[
E'_m \colon y^2 = x^3+x-m^6, \quad \text{where $m\equiv 2 \mod 4$ and $3 \nmid m$}.
\]
We first observe that $E'_m(\Q)$ contains the two rational points
\[
P'=(m^2,m) \quad \text{and} \quad Q'=(m^6,m^9).
\]
To prove Theorem \ref{thm:main-intro-rk2}, it suffices to show that $P'$ and $Q'$ are independent in $E'_m(\Q)$. Just as before, will do so by showing that they satisfy the conditions of Proposition \ref{prop:utility-of-trivial-torsion}. Thus, we must show that $P'$, $Q'$, and $P'+Q'$ are nontrivial in $E'_m(\Q)/2E'_m(\Q)$.

We begin by with a lemma which will allow us to easily deduce handle the cases of $P'$ and $Q'$.

\begin{lemma}\label{lem:family2prop1}
Let $A$ and $B$ be rational points in $E'_m(\Q)$. Suppose that $x(A)=a^{i}$ where
\begin{enumerate}
\item $a>1$ is an integer such that $3 \nmid a$, and
\item $i$ is an even positive integer.

\end{enumerate}
Then, $A \neq 2B$.
\end{lemma}

\begin{proof}
Just as before, we may write $x(B)=\frac{u}{r^2}$ where $u,r, \in \Z$ with $\gcd(u,r)=1$. Suppose $A=2B$; then the doubling formula \eqref{eq:doubling-formula} yields
\begin{equation}\label{eq:a^i-equation}
a^i(4u^3r^2+4ur^6-4m^6r^8) \equiv u^4 - 2u^2r^4+8um^6r^6 + r^8
\end{equation}
Suppose $r \neq \pm 1$. Then $p|r$ for some prime $p$. Considering \eqref{eq:a^i-equation} modulo $p$ gives the congruence
\[
0 \equiv u^4 \mod p.
\]
Thus $p|u$, but as $\gcd(u,r) = 1$, this is a contradiction. So we may assume $r = \pm 1$, and \eqref{eq:a^i-equation} simplifies to
\begin{equation}\label{eq:family2rpm1}
4a^i(u^3+u-m^6) = u^4 - 2u^2+8um^6 + 1.
\end{equation}
As $3 \nmid a$ and $3\nmid m$, we have $$a^i \equiv m^6 \equiv 1 \mod 3.$$ Thus, if we consider \eqref{eq:family2rpm1} modulo 3, we get
\[
u^3 + u - 1 \equiv u^4 - 2u^2 + 2u + 1 \mod 3,
\]
and after rearranging this becomes
\[
-u(u^3-u^2+u+1)\equiv 2 \mod 3.
\]
But the left-hand side of this congruence can only ever be $0$ or $1$, so we have obtained a
contradiction. We conclude that $A \neq 2B$.
\end{proof}

We now apply this lemma to $P'$ and $Q'$.

\begin{proposition}\label{prop:family2prop1}
The points $P'$ and $Q'$ are nontrivial in $E'_m(\Q)/2E'_m(\Q)$.
\end{proposition}
\begin{proof}
Recall that $x(P')=m^2$ and $x(Q')=m^6$, where $m>1$ is an odd integer such that $3 \nmid a$. The result is thus immediate from Lemma \ref{lem:family2prop1}.
\end{proof}

It remains to study the point
$P'+Q'$.

\begin{proposition}\label{prop:family2prop-pq}
The point
\[
P'+Q'=\left( \frac{m^4+1}{m^2}, \frac{-2m^4 - 1}{m^3} \right).
\]
is nontrivial in $E'_m(\Q)/2E'_m(\Q)$.
\end{proposition}

\begin{proof}
Suppose $P'+Q'=2A$ for some $A\in E(\Q)$, and write $x(A)=\frac{u}{r^2}$ where $\gcd(u,r)=1$. By \eqref{eq:doubling-formula}, we must have
\begin{equation}\label{eq:pq-double}
-4r^2(m^4+1)(m^6r^6 - ur^4 - u^3)
=
(r^4 -u^2)^2+ 8m^6ur^6.
\end{equation}
Since $m \equiv 2 \mod 4$, reducing considering \eqref{eq:pq-double} modulo 16 gives
\[
4r^2u(r^4+u^2)\equiv 4(r^4-u^2)^2 \mod 16,
\]
which in turn implies
\begin{equation}\label{pq-double-reduced-mod-4}
r^2u(r^4+u^2)\equiv (r^4-u^2)^2 \mod 4.
\end{equation}

If $r$ is even, then \eqref{pq-double-reduced-mod-4} implies
\[
0 \equiv u^4 \mod 4,
\]
which implies $u$ is also even, contradicting $\gcd(u,r)=1$. We obtain a similar contradiction if we assume $u$ is even; thus we see that both $u$ and $r$ must be odd.

Thus $r^2 \equiv u^2 \equiv 1 \mod 4$, and \eqref{pq-double-reduced-mod-4} becomes
\[
2u \equiv 0 \mod 4,
\]
contradicting the fact that $u$ must be odd. We have thus proved that $P'+Q' \neq 2A$ as desired.
\end{proof}

\begin{proof}[Proof of Theorem \ref{thm:main-intro-rk2}]
Propositions \ref{prop:family2prop1} and \ref{prop:family2prop-pq} show that $P'$ and $Q'$ generate a subgroup of order $4$ in $E'_m(\Q)/2E'_m(\Q)$, so the result follows from Proposition \ref{prop:utility-of-trivial-torsion}.
\end{proof}

\section{Studying the families $E_m$ and $E'_m$ via the canonical height matrix}\label{sec:height-matrix}

In this final section, we once again study the families $E_m$ and $E_m'$, but we will assume only that $m>1$ is an integer. We show via less elementary means that the points
\[
\{P,Q,R\} \subseteq E_m \quad \text{and} \quad \{P',Q'\} \subseteq E_m'
\]
are linearly independent. By Theorem \ref{thm:torsion-family-two} both of these points are of infinite order, so it then follows that the rank of $E(\Q)$ is at least two. In order to prove the independence of $P$ and $Q$, we will use the N\'{e}ron-Tate canonical height pairing matrix, which we now recall for the reader. Our exposition will closely follow that of \cite[$\S$2.7]{alvaro-book}.

\subsection{N\'{e}ron-Tate height, pairing, and matrix} In this section we will describe some generalities which are not specific to our particular curve $E'_m$, and so we write $E/\Q$ for any elliptic curve, and we write $P,Q \in E(\Q)$ for any two rational points on that curve.

Given a rational number $x=\frac{p}{q}$, its \textit{height} is defined as
\[
\log(\max\{|p|,|q| \}).
\]
For any elliptic curve $E/\Q$, this extends to a height function on its rational points via application to the $x$-coordinate:
\begin{align*}
    H \colon E(\Q) &\rightarrow \mathbf{R}\\ P &\mapsto h(x(P)).
\end{align*}
Then for $P \in E(\Q)$, the \textit{N\'{e}ron-Tate  height} (also called the \textit{canonical height}) is defined by
\[
\hat{h}(P)=\frac{1}{2} \lim_{N \to \infty} \frac{H(2^N \cdot P)}{4^N}.
\]
The utility of this normalization procedure is that it produces a height function with very desirable properties, as explained by the next proposition.

\begin{proposition}\label{prop:Neron-tate-height}
Let $E/\Q$ be an elliptic curve and $\hat{h}$ be the canonical height on $E$.
\begin{enumerate}
    \item For all $P,Q \in E(\Q)$, we have the parallelogram law:
    \[
    \hat{h}(P+Q)+\hat{h}(P-Q) = 2 \hat{h}(P) + 2 \hat{h}(Q).
    \]
    \item For all $P \in E(\Q)$ and $m \in \Z$, we have
    \[
    \hat{h}(mP)=m^2 \hat{h}(P).
    \]
    \item Let $P \in E(\Q).$ Then $\hat{h}(P)\geq0$, and $\hat{h}(P)=0$
 if and only if $P$ is a torsion point.
\end{enumerate}
\end{proposition}
We refer the reader to \cite[Chapter VIII, Theorem 9.3]{silverman-arithmetic} for the proof of this proposition.

We may now define the \textit{N\'{e}ron-Tate pairing} by
\begin{align*}
\langle \cdot , \cdot \rangle &\colon E(\Q) \times E(\Q) \to \mathbf{R}\\
\langle P, Q \rangle &= \hat{h}(P+Q)-\hat{h}(P)-\hat{h}(Q).
\end{align*}
This pairing can be shown to be a non-degenerate symmetric bilinear form on $E(\Q)/E(\Q)_{\mathrm{tors}}$ \cite[Theorem 2.8.5]{alvaro-book}. As a consequence of this property, we have the following very useful result.

\begin{proposition}\label{prop:matrix}
Let $E/\Q$ be an elliptic curve and let $P_1,\cdots, P_n \in E(\Q)$ be points of infinite order. Define the height matrix $\mathcal{H}_{P_1,\cdots, P_n}$ by
\[
\mathcal{H}_{P_1,\cdots, P_n}=\left(
\begin{matrix}
\langle P_1, P_1 \rangle & \langle P_1, P_2 \rangle & \cdots & \langle P_1,P_n\rangle \\
\langle P_2, P_1 \rangle & \langle P_2, P_2 \rangle & \cdots & \langle P_2,P_n\rangle \\
\vdots & \vdots & \ddots & \vdots \\
\langle P_n, P_1 \rangle & \langle P_n, P_2 \rangle & \cdots & \langle P_n,P_n\rangle
\end{matrix}
\right)
\]
Then $P_1, \ldots, P_n$ are independent in $E(\Q)$ if and only if $\det \mathcal{H}_{P_1,\cdots, P_n}\neq0$.
\end{proposition}

\begin{proof} See \cite[Corollary 2.8.6]{alvaro-book}.
\end{proof}

We note that, since the N\'{e}ron-Tate pairing is symmetric, we have $\langle P_i, P_j \rangle = \langle P_j, P_i \rangle$.

\subsection{Computing the height matrix for $P',Q' \in E'_m(\Q)$}
As a concrete example, we will now apply the results of the previous section to our points
\[
P'=(m^2,m) \quad \text{and} \quad Q'=(m^6,m^9)
\]
on the curve $E'_m(\Q)$, where we now assume only that $m > 1$ is an integer.

Before diving into our computations, we note that, by the base change formula for logarithms, it is sufficient for us to take $\log = \log_m$ in our computations of $\hat{h}$. [This explains our assumption that $m>1$ despite our proof of Theorem \ref{thm:torsion-family-two} working for all nonzero $m$.] By making this choice, the computations become relatively straightforward, as we now explain.

For a point $A=(x_0,y_0) \in E(\Q)$, the doubling formula \eqref{eq:doubling-formula}, for our curve $E'_m$ specializes to
\begin{equation}\label{eq:doubling-formula-specialized}
x(2A) = \frac{x_0^4 - 2x_0^2 + 8x_0m^6 + 1}{4(x_0^3+x_0-m^6)}.
\end{equation}
Let us consider $Q'=(m^6,m^9)$, and write
\[
2^N \cdot Q' = \frac{p_N}{q_N}.
\]
Then \eqref{eq:doubling-formula-specialized} implies that for all $N \geq 1$,
\begin{align*}
H(2^N \cdot Q') &= \log (\max\{|p_N|,|q_N|\}) \\
&= \log(p_N)\\
&= \log\left( m^{6 \cdot 4^N} + g_N (m) \right)
\end{align*}
where $g_N(m)$ is a polynomial in $m$ of degree less than $6 \cdot 4^N$. It follows that
\begin{align*}
\hat{h}(Q') &=  \frac{1}{2} \lim_{N \to \infty} \frac{H(2^N \cdot Q')}{4^N} \\
&=\frac{1}{2} \lim_{N \to \infty} \frac{6 \cdot 4^N}{4^N} \\
&=\frac{1}{2} \cdot 6\\
&=3.
\end{align*}

Similarly, one easily computes \[
\hat{h}(P')=1 \quad \text{and} \quad \hat{h}(P'+Q')=2.
\]
Now using Proposition \ref{prop:Neron-tate-height} we can compute the entries of the height pairing matrix $\mathcal{H}_{P',Q'}$. We have
\begin{align*}
\langle P',P' \rangle &= \hat{h}(2\cdot P') - 2 \hat{h}(P') \\
&=2^2 \hat{h}(P') - 2 \hat{h}(P')\\
&=2\hat{h}(P')\\
&=2.
\end{align*}
Similarly, one computes
\begin{align*}
\langle Q', Q' \rangle &= 2 \hat{h}(Q') \\
&=4
\end{align*}
and
\begin{align*}
\langle P',Q' \rangle &= \hat{h}(P'+Q') - \hat{h}(P') - \hat{h}(Q')\\ &=2-1-3\\ &= -2.
\end{align*}
Since the N\'{e}ron-Tate pairing is symmetric, we also have $\langle Q', P' \rangle=-2.$

We may now prove Theorem \ref{thm:unconditional}.

\begin{proof}[Proof of Theorem \ref{thm:unconditional}]
For the family $E_m'$, the computations preceding this proof immediately imply
\[
\mathcal{H}_{P',Q'} = \left(
\begin{matrix}
2 & -2 \\
-2 & 6
\end{matrix}
\right).
\]
We note that this matrix does not depend on the value of $m>1$. Since it has nonzero determinant, Proposition \ref{prop:matrix} tells us that $P'$ and $Q'$ are independent in $E'_m(\Q)$, and so $E'_m(\Q)$ has rank at least two.

We may perform precisely the same type of calculations for the points $P,Q,$ and $R$ on $E_m$ to obtain
\[
\mathcal{H}_{P,Q,R} = \left(
\begin{matrix}
2 & -1 & 0\\
-1 & 2 & -1 \\
0 & -1 & 2
\end{matrix}
\right).
\]
Since this matrix is invertible, we see that Theorem \ref{thm:main-intro-rk3} also holds assuming only that $m>1$.
\end{proof}

\begin{remark}
For many values of $m$, one checks with Sage \cite{sage} that $E'_m(\Q)$ has rank larger than $2$. For instance, the rank is $3$ when $m=4$ or $6$, and the rank is $4$ when $m=7$ or $10$. It would be interesting to explicitly describe subfamilies of $E'_m$ with rank larger than two. In fact, some cases have already been explored in the work of \cite{tadic}.
\end{remark}

\begin{remark} Although we chose not to pursue it here, it should be possible to further strengthen Theorem \ref{thm:unconditional} in several ways. In particular, Silverman's Specialization Theorem should imply that the conclusion of Theorem \ref{thm:unconditional} holds for all but finitely many \textit{rational} values of $m$.
\end{remark}
\bibliographystyle{amsalpha}
\bibliography{references}
\end{document}